\documentclass[12pt]{amsart}
\usepackage{times}
\usepackage{amsfonts, amsmath, amssymb, amsthm, mathtools, dsfont}
\usepackage[margin=2cm]{geometry}
\usepackage[colorlinks=true,citecolor=blue]{hyperref}
\usepackage{setspace, enumerate}
\usepackage{comment}
\usepackage[title]{appendix}
\usepackage{todonotes}

\theoremstyle{plain}
\newtheorem{theorem}{Theorem}[section]
\newtheorem{proposition}[theorem]{Proposition}
\newtheorem{lemma}[theorem]{Lemma}

\newtheorem{corollary}[theorem]{Corollary}
\newtheorem{claim}[theorem]{Claim}
\newtheorem{question}[theorem]{Question}

\newcommand{\st}{\text{st}}
\newcommand{\cl}{\text{cl}}

\theoremstyle{definition}

\newtheorem{remark}[theorem]{Remark}

\title{Topology of independence complexes and cycle structure of hypergraphs}

\author{Jinha Kim}
\address{Department of Mathematics\\
Chonnam National University\\
Gwangju\\
Republic of Korea 
and Discrete Mathematics Group\\
Institute for Basic Science (IBS)\\
Daejeon\\
Republic of Korea}
\email{jinhakim@jnu.ac.kr}



\date{\today}

\begin{document}
\maketitle
\begin{abstract}
Recently, Zhang and Wu proved a conjecture of Kalai and Meshulam, showing that for every graph $G$ without induced cycles of length divisible by $3$, the sum of all reduced Betti numbers of its independence complex $I(G)$ is at most $1$. We extend this result to the hypergraph setting. Namely, we show that the same conclusion holds for any hypergraph $H$ that does not contain a Berge cycle of length divisible by $3$. This establishes a broader connection between forbidden cycle structures and the topological simplicity of independence complexes. As a key tool, we introduce a hypergraph analogue of Barmak's star cluster theorem for graphs. This new theorem implies, in particular, that if a hypergraph $H$ has a vertex $v$ that is not isolated and is not contained in an induced Berge cycle of length $3$, then there exists a hypergraph $H'$ with fewer vertices than $H$ such that the independence complex of $H$ is homotopy equivalent to the suspension of the independence complex of $H'$.
\end{abstract}

\section{Introduction}
Throughout the paper, we consider only finite vertex set for graphs, hypergraphs and simplicial complexes.
An \textit{independent set} of a graph $G$ is a vertex subset that does not contain any edge.
The \textit{independence complex} $I(G)$ of $G$ is the collection of all independent sets of $G$.
The class of independence complexes of graphs is the same as the class of flag complexes, the simplicial complexes where all minimal non-faces have size $2$.
In general, for any finite simplicial complex $X$, there is an independence complex $I(G)$ that is homotopy equivalent to $X$. 
To see this, consider the barycentric subdivision $\text{sd}(X)$ of $X$.
Clearly, $\text{sd}(X)$ is homotopy equivalent to $X$.
On the other hand, $\text{sd}(X)$ is a flag complex, and hence it equals to an independence complex of some graph.
Thus homotopy types and homology groups of independence complexes of graphs are as various as those of simplicial complexes.

Many results have been obtained on topological properties of independence complexes of graphs, especially on the homotopy types and homology groups of independence complexes of special classes of graphs.
For example, Kozlov~\cite{Koz99} determined the homotopy types of independence complexes of paths and cycles, and Ehrenborg and Hetyei~\cite{EH06} showed the homotopy types of independence complexes of forests are either contractible or a sphere.
Furthermore, Kim~\cite{Kim22} proved that the homotopy type of the independence complex of a graph with no induced cycles of length divisible by $3$ are also either contractible or a sphere, strengthening the result by Engstr\"{o}m~\cite{Eng20}. 
See also \cite{Kawa10} for chordal graphs, \cite{BLN08} for square grids, and \cite{Braun11} for stable Kneser graphs.

Meanwhile, Barmak~\cite{Barmak} introduced the notion of star clusters in independence complexes of graphs, which turned out to be useful to understand the independence complexes of graphs. He actually gave alternative proofs of several known results on the independence complexes of graphs by simpler arguments using his star cluster theorem: see Theorem~\ref{thm:starcluster}.
We generalize Barmak's result, providing an analogous statement for the independence complexes of hypergraphs, and then relate it with a hypergraph operation.
Using the star cluster theorems, we also establish a hypergraph analogue of a weaker version of the result in \cite{Kim22}, showing that the boundedness of the total Betti number of the independence complex of a hypergraph with certain forbidden cycle structure.

\subsection{Basic terminology about hypergraphs}
A \textit{hypergraph} $H$ on $V$ is a collection of nonempty subsets of $V$, where each element of $H$ is called an \textit{edge} of $H$.
Note that a graph is a hypergraph whose edges have size exactly $2$.
We denote by $V(H)$ the vertex set of $H$.
The independent sets and independence complexes of hypergraphs can be defined in analogous way.
An \textit{independent set} of a hypergraph $H$ is a vertex subset that does not contain an edge as a subsets, and the \textit{independence complex} $I(H)$ of $H$ is a collection of all independent sets of $H$. That is,
\[I(H):=\{W \subset V(H) \mid W \text{ contains no edge of } H\}.\]

A vertex $v \in V(H)$ is {\em isolated} in $H$ if it is not contained in any edge of $H$. Observe that, if $v$ is isolated in $H$, then $I(H)$ is a cone with apex $v$, thus is contractible.

If $H$ has two edges $f$ and $g$ with $g \subset f$, then we clearly have $I(H)=I(H-\{f\})$.
In addition, if $\{v\}$ is an edge of $H$ for a vertex $v$, then $I(H)=I(H-v)$ where $H-v$ is the hypergraph obtained from $H$ deleting the vertex $v$ and all edges containing $v$.
Therefore, unless stated otherwise, we assume that every edge in a hypergraph is inclusion-minimal and has size at least $2$.

Now we consider ``cycles" in hypergraphs.
Among various possible definitions of cycles in hypergraphs, we consider ``Berge cycle", which is broadest concept of cycles.
 \begin{itemize}
        \item A \textit{Berge cycle} of length $k$ in a hypergraph $H$ is an alternating sequence $C=v_1 e_1 v_2 e_2 \cdots v_k e_k$ of $k$ distinct vertices $v_1,v_2,\ldots,v_k$ and $k$ distinct edges $e_1,e_2,\ldots,e_k$ of $H$ where $v_i,v_{i+1} \in e_i$ for all index $i$ chosen modulo $k$. Note that in graphs, Berge cycles are equal to usual cycles.
        \item A Berge cycle $C=v_1 e_1 v_2 e_2 \cdots v_k e_k$ in $H$ is \textit{induced} if $e_i \cap \{v_1,v_2,\ldots,v_k\}=\{v_i,v_{i+1}\}$ for all $i $ chosen modulo $k$ and there is no edge of $H$ containing two non-consecutive vertices of $C$.
        \item Two Berge cycles $C_1=v_1f_1v_2f_2\ldots v_kf_k$, $C_2=w_1 g_1 w_2 g_2 \ldots w_{\ell} g_{\ell}$ are \textit{disjoint} if $\{v_1,\ldots,v_k\} \cap \{w_1,\ldots,w_{\ell}\} =\varnothing$  and $\{f_1,\ldots,f_k\} \cap \{g_1,\ldots,g_{\ell}\} =\varnothing$.
For a hypergraph $H$, let $t(H)$ be the maximum integer $k$ such that $H$ has $k$ pairwise disjoint ternary Berge cycles.
Note that $t(H)=0$ if and only if $H$ has no ternary Berge cycle.
    \end{itemize}

\subsection{Total Betti numbers of independence complexes of hypergraphs}
Our main theorem concerns the Betti numbers of independence complexes of certain hypergraphs. It can be derived as an application of the star cluster theorem, which will be introduced later.
This is inspired by a conjecture by Kalai and Meshulam~\cite{KM} for graphs, that suggests a connection between the topology of independence complexes and the cycle structure of graphs.
They conjectured that for a graph with no induced cycle of length $0$ modulo $3$, the number of independent sets of odd size and the number of independent sets of even size differ by at most $1$. Such a graph is said to be {\em ternary}. This conjecture was proved in \cite{CSSS20} by a purely combinatorial argument.

Indeed, the difference between the number of independent sets of odd size and the number of independent sets of even size in a graph $G$ is actually equal to the {\em reduced Euler characteristic} $\tilde{\chi}(I(G))$ of $I(G)$, which is a well-known topological invariant of $I(G)$.
In this respect, Kalai and Meshulam also conjectured a stronger version that the sum of all reduced Betti numbers of $I(G)$ is at most $1$ for any ternary graph $G$. This stronger conjecture was proved in \cite{WZ25}, and even a stronger statement was proved in \cite{Kim22}, arguing that the homotopy type of $I(G)$ of a ternary graph $G$ is either contractible or homotopy equivalent to a sphere.

Going back further than the presentation of the Kalai--Meshulam's conjectures, the following result by Kozlov~\cite{Koz99} on independence complexes of cycles already showed a phenomenon related to modulo $3$.
That is, if $C_{\ell}$ is a cycle of length $\ell$, then
\[I(C_{\ell}) \simeq
\begin{cases}
    \mathbb{S}^k \vee \mathbb{S}^k & \text{ if } \ell=3k+3,\\
    \mathbb{S}^k & \text{ if } \ell=3k+2,3k+4,
\end{cases}\]
where $\mathbb{S}^k$ is the $k$-dimensional sphere and $X \vee Y$ is the wedge sum of two topological spaces of $X$ and $Y$.
Furthermore, Csorba~\cite{Csorba} proved that if $\tilde{G}$ is a graph obtained from $G$ by replacing an edge by a path of length $4$, then $I(\tilde{G}) \simeq \Sigma I(G)$, which  directly explains the above Kozlov's result since $\Sigma \mathbb{S}^k \simeq \mathbb{S}^{k+1}$.
The result by Csorba is actually a special case of the star cluster theorem.
Thus it is natural to suspect that the star cluster theorem can be related to the Kalai--Meshulam conjecture.

Considering hypergraphs, we prove the following theorem using the star cluster theorem.
\begin{theorem}\label{thm:main2}
Let $H$ be a hypergraph with no Berge cycle of length divisible by $3$.
Then the sum of all reduced Betti numbers of $I(H)$ is at most $1$.
\end{theorem}

\subsection{Star clusters in hypergraphs}
We now introduce the star cluster theorem, which serves as a central part in the proof of Theorem~\ref{thm:main2}.
We have the following theorem for star clusters in independence complexes of hypergraphs, which generalizes Barmak's result \cite{Barmak}.
Here $\Sigma X$ is the suspension of a simplicial complex $X$ and $\st(\sigma)$ of $I(H)$ is the subcomplex whose maximal faces are precisely the maximal faces of $I(H)$ that contain $\sigma$.
\begin{theorem}[Star cluster theorem]\label{thm:main}
Let $H$ be a hypergraph and $v$ be a vertex of $H$ such that $v$ is not isolated and is not contained in an induced Berge cycle of length $3$.
Then
\[I(H) \simeq \Sigma (\st(v) \cap (\bigcup_{i=1}^{k} \st(\tilde{e}_i))),\]
where $e_1, e_2, \ldots, e_k$ are the edges of $H$ containing $v$ and $\tilde{e}_i=e_i\setminus\{v\}$.
\end{theorem}

One important remark is that every simplicial complex can be given as an independence complex of some hypergraph. 
For a simplicial complex $X$ on $V$, let $H_X$ be the hypergraph on $V$ where the edges of $H_X$ are the minimal non-faces of $X$. 
Then it is clear that $X$ is equal to the independence complex of $H_X$.
Thus Theorem~\ref{thm:main} can be interpreted to a result for arbitrary simplicial complexes.

Also, we can deduce the following from Theorem~\ref{thm:main}.
\begin{theorem}\label{thm:cor}
Let $H$ be a hypergraph and $v$ be a vertex of $H$ such that $v$ is not isolated and is not contained in an induced Berge cycle of length $3$.
\begin{enumerate}
    \item There is a hypergraph $H'$ defined on $V(H)\setminus\{v\}$ such that $I(H) \simeq \Sigma I(H')$.
    \item $I(H)$ has the homotopy type of a suspension.
    \item $\tilde{H}_1(I(H))$ is a free abelian group.
\end{enumerate}
\end{theorem}

\subsection*{Organization}
In Section~\ref{sec:pre}, we introduce basic terminologies and topological backgrounds. We present the proof of Theorem~\ref{thm:main} in Section~\ref{sec:main_pf} and introduce a specific structural condition for Theorem~\ref{thm:main} and discuss some examples in Section~\ref{sec:structure}.
In Section~\ref{sec:KM}, we prove Theorem~\ref{thm:main2} and also discuss a relation between the cycle structure and bounds on the total Betti number, using the Erd\H{o}s--Posa theorem for cycles of length $0$ modulo $3$.

\section{Preliminary}\label{sec:pre}
This section provides basic terminology and the necessary background about homotopy and homology in algebraic topology, all of which is standard graduate-level material; we refer to \cite{Hat} for a comprehensive reference. Readers already familiar with these topics may proceed to the next section.

A(n) \textit{(abstact) simplicial complex} $X$ on a vertex set $V$ is a collection of subsets of $V$ that is closed under taking subsets.
Each element of a simplicial complex $X$ is called a \textit{face} of $X$, and a \textit{missing face} of $X$ is an inclusion-minimal non-face of $X$.

For a simplicial complex $X$, we denote by $\tilde{H}_i(X)$ the \textit{$i$-th (reduced) homology group} of $X$ over $\mathbb{Z}$.
The \textit{$i$-th (reduced) Betti number} of $X$ is the integer $\tilde{\beta}_i(X)=\text{rank}(\tilde{H}_i(X))$, and the \textit{total Betti number} $\beta(X)$ of $X$ is the sum of all reduced Betti numbers of $X$.

We denote by $X \simeq Y$ if two topological spaces $X$ and $Y$ are homotopy equivalent. 
For a topological space $X$, the \textit{suspension} $\Sigma X$ of $X$ is the quotient space
\[\Sigma X := X\times [0,1]/(X\times\{0\})/(X\times \{1\}),\] 
which means $\Sigma X$ is obtained from $X\times [0,1]$ by identifying each of $X\times\{0\}$ and $X\times\{1\}$ as single points.
For example, $\Sigma \mathbb{S}^n \simeq \mathbb{S}^{n+1}$ for the $n$-dimensional sphere $\mathbb{S}^n$.

\begin{proposition}\label{prop1}
For a simplicial complex $X$, $\tilde{H}_{i+1}(\Sigma X) \cong \tilde{H}_i(X)$ for all $i$.
Hence, $\beta(\Sigma X)=\beta(X)$.
\end{proposition}

\begin{proposition}\label{prop2}
Let $K, K_1, K_2$ be simplicial complexes such that $K=K_1 \cup K_2$. Then 
   \begin{enumerate}[(a)]
    \item $K/K_2 \simeq K_1 /(K_1 \cap K_2)$, and
    \item for simplicial complexes $M$ and $L$, if $L$ is contractible in $M$, then $M/L \simeq M \vee \Sigma L$.
    In particular, $M/L \simeq \Sigma L$ if $M$ is contractible, and $M/L \simeq M$ if $L$ is contractible.
\end{enumerate}
By combining (a) and (b), we also obtain the following.
\begin{enumerate}
    \item[(c)] Assume both $K_1$ and $K_2$ are contractible.
    Then $$K \simeq K/K_2 \simeq K_1/(K_1 \cap K_2) \simeq \Sigma (K_1 \cap K_2).$$
    If $K_1 \cap K_2$ is contractible, then $K$ is also contractible.
\end{enumerate} 
\end{proposition}

When we study simplicial complxes, one of the most frequently used tool is the Mayer--Vietoris exact sequence.
Let $A,B$ be simplicial complexes and $K=A \cup B$. Then the \textit{Mayer--Vietoris sequence} for $(K,A,B)$ is the following long exact sequence:
\begin{equation}\label{mvseq}
    \cdots \to \tilde{H}_i (A \cap B) \to \tilde{H}_i(A) \oplus \tilde{H}_i(B) \xrightarrow{\gamma_i} \tilde{H}_i(K) \xrightarrow{\delta_i} \tilde{H}_{i-1} (A \cap B) \to \cdots.
\end{equation}
The following inequality about the Betti numbers, which will be used in Section~\ref{sec:KM}, follows from \eqref{mvseq}.
\begin{lemma}\label{mvseq-betti}
For simplicial complexes $A,B$ and $K=A \cup B$, the following holds for all $i$:
\[\tilde{\beta}_i(K) \leq \tilde{\beta}_i(A)+\tilde{\beta}_i(B)+\tilde{\beta}_{i-1}(A\cap B).\]
\end{lemma}
\begin{proof}
Consider the Mayer--Vietoris sequence~\eqref{mvseq} for $(K,A,B)$.
Since \eqref{mvseq} is exact, we obtain the following short exact sequence for each $i$:
$$0 \to \text{ker}\delta_i \hookrightarrow \tilde{H}_i(K) \xrightarrow{\delta_i} \text{im}\delta_i \to 0.$$
Then we obtain $\tilde{\beta}_i(K)=\text{rk}(\text{ker}\delta_i)+\text{rk}(\text{im}\delta_i)
=\text{rk}(\text{im}\gamma_i)+\text{rk}(\text{im}\delta_i) \leq \tilde{\beta}_i(A)+\tilde{\beta}_i(B)+\tilde{\beta}_{i-1}(A\cap B).$
\end{proof}

\section{Proof of Theorem~\ref{thm:main} : star cluster theorem}\label{sec:main_pf}

In this section, we prove Theorem~\ref{thm:main}.
Before that, we first recall Barmak's theorem on star clusters in independence complexes of graphs.

For a graph $G$ and a vertex $v$ of $G$, the \textit{neighborhood} $N_G(v)$ of $v$ in $G$ is the set of vertices adjacent to $v$ in $G$ and $N_G[v]:=N_G(v) \cup \{v\}$.
For a vertex subset $W$ of $G$, the \textit{neighborhood} $N_G(W)$ of $W$ in $G$ is $N_G(W):=\bigcup_{w \in W}N_G(w)$ and $N_G[W]:=N_G(W)\cup W$.
Let $X$ be a simplicial complex and $\sigma$ be a face of $X$.
The \textit{star} $\text{st}_X(\sigma)$ of $\sigma$ in $X$ is defined as
\[\text{st}_X (\sigma)=\{\tau \subset V \mid \tau \cup \sigma \in X\}.\]
Note that $\text{st}_X(\sigma)$ is always contractible.
The \textit{star cluster} $\text{SC}_X(\sigma)$ of $\sigma$ in $X$ is defined as $$\text{SC}_X(\sigma)=\bigcup_{v \in \sigma}\text{st}_X(v).$$
We write $\text{st}_X(\sigma)=\text{st}(\sigma)$ and $\text{SC}_X(\sigma)=\text{SC}(\sigma)$ if it is clear which ground complex we are dealing with.

\begin{theorem}[\cite{Barmak}, Star cluster theorem for graphs]\label{thm:starcluster}
    Let $G$ be a graph and $v$ be a vertex that is not isolated and is not contained in a triangle.
    Then 
    \[I(G) \simeq \Sigma (\st(v) \cap SC(N_G(v))).\]
\end{theorem}

The above theorem is useful to understand the independence complex of graphs, as $\st(v) \cap SC(N_G(v))$ is a simplicial complex defined on the vertex set $V(G)-N_G[v]$ that has fewer vertices than $I(G)$.
However, it is not guaranteed that $\st(v) \cap SC(N_G(v))$ is an independence complex of a graph.
For this reason, Theorem~\ref{thm:starcluster} has limitation in relating the topology of independence complexes and graph structure.
On the other hand, $\st(v) \cap SC(N_G(v))$ can be viewed as the independence complex of some hypergraph, and this leads us to invent the following hypergraph analogue of Theorem~\ref{thm:starcluster}.

\begingroup
\def\thetheorem{\ref{thm:main}}
\begin{theorem}
Let $H$ be a hypergraph and $v$ be a vertex of $H$ such that $v$ is not isolated and is not contained in an induced Berge cycle of length $3$.
Then
\[I(H) \simeq \Sigma (\st(v) \cap (\bigcup_{i=1}^{k} \st(\tilde{e}_i))),\]
where $e_1, e_2, \ldots, e_k$ are the edges of $H$ containing $v$ and $\tilde{e}_i=e_i\setminus\{v\}$.
\end{theorem}
\addtocounter{theorem}{-1}
\endgroup

Note that Theorem~\ref{thm:starcluster} can be obtained from Theorem~\ref{thm:main} by replacing $H$ with a graph.
In order to prove Theorem~\ref{thm:main}, We first prove the following lemma.
\begin{lemma}\label{lem1}
Let $X$ be a simplicial complex on $V$.
Let $\sigma \in X$ and $\sigma_1, \sigma_2, \ldots, \sigma_k \subset \sigma$.
Assume that if $u \in \sigma_i \setminus \sigma_j$ and $w \in \sigma_j\setminus\sigma_i$ for some $i,j \in [k]$, then all missing faces containing both $u$ and $w$ have a vertex not in $\st_X(\sigma_i) \cup \st_X(\sigma_j)$.
Then $\bigcup_{i=1}^{k} \st_X(\sigma_i)$ is contractible.
\end{lemma}

\begin{proof}
We prove it by the induction on $k$.
If $k=1$, then it is trivial.
Suppose $k>1$. Let $X_1=\bigcup_{i=1}^{k-1}\st_X(\sigma_i)$ and $X_2=\st_X(\sigma_k)$.
By the induction hypothesis, both $X_1$ and $X_2$ are contractible.
To show $X_1 \cup X_2$ is contractible, it is enough to show $X_1 \cap X_2$ is contractible by Proposition~\ref{prop2} (c).
Note that $X_1 \cap X_2=\bigcup_{i=1}^{k-1}(\st_X(\sigma_i) \cap \st_X(\sigma_k))$.

\begin{claim}\label{claim1}
$\st_X(\sigma_i) \cap \st_X(\sigma_k)=\st_X(\sigma_i \cup \sigma_k)$ for each $i \in [k-1]$.
\end{claim}
\begin{proof}[Proof of Claim~\ref{claim1}]
We first show the easy direction that $\st_X(\sigma_i) \cap \st_X(\sigma_k) \supset \st_X(\sigma_i \cup \sigma_k)$.
If $\eta \in \st_X(\sigma_i \cup \sigma_k)$, then $\eta \cup \sigma_i \cup \sigma_k \in X$. 
Hence, we obtain $\eta \cup \sigma_i, \eta \cup \sigma_k \in X$, which is equivalent to $\eta \in \st_X(\sigma_i) \cap \st_X(\sigma_k)$.

Now for the opposite direction, take $\tau \in \st_X(\sigma_i) \cap \st_X(\sigma_k)$. 
We want to show $\tau \in \st_X(\sigma_i \cup \sigma_k)$, i.e. $\tau \cup \sigma_i \cup \sigma_k \in X$.
Suppose $\tau \cup \sigma_i \cup \sigma_k \not\in X$.
Then we can take a missing face $S$ of $X$ in $\tau \cup \sigma_i \cup \sigma_k$.
Note that $\tau \cup \sigma_i, \tau \cup \sigma_k \in X$ since $\tau \in \st_X(\sigma_i) \cap \st_X(\sigma_k)$.
Then $S \setminus (\tau \cup \sigma_i) \neq \varnothing$ and $S \setminus (\tau \cup \sigma_k) \neq \varnothing$.
Thus we can take $u \in S \setminus (\tau \cup \sigma_i) \subset \sigma_k \setminus \sigma_i$ and $w \in S \setminus (\tau \cup \sigma_k) \subset \sigma_i \setminus \sigma_k$.
Then by the assumption, $S$ must contain a vertex $x$ not in $\st_X(\sigma_i) \cup \st_X(\sigma_k)$.
Since $x \in S \subset \tau \cup \sigma_i \cup \sigma_k$, we obtain that either $x \in \tau \cup \sigma_i$ or $x \in \tau \cup \sigma_k$.
However, it implies that $x$ is a vertex of $\st_X(\sigma_i) \cup \st_X(\sigma_k)$ since $\tau \cup \sigma_i, \tau \cup \sigma_k \in X$.
Thus we reach a contradiction and this completes the proof.
\renewcommand{\qedsymbol}{$\blacksquare$}
\end{proof}

Now, by Claim~\ref{claim1}, we have $X_1 \cap X_2=\bigcup_{i=1}^{k-1} \st_X(\sigma_i \cup \sigma_k)$.
To use the induction hypothesis, we will show that $\sigma_1 \cup \sigma_k, \sigma_2 \cup \sigma_k, \ldots, \sigma_{k-1} \cup \sigma_k$ satisfy the assumption in the statement.
First, it is obvious that each $\sigma_i \cup \sigma_k$ is contained in $\sigma$ for $i \in [k-1]$.
Furthermore, if $u \in (\sigma_i \cup \sigma_k)\setminus(\sigma_j \cup \sigma_k) \subset \sigma_i\setminus\sigma_j$ and $w \in (\sigma_j \cup \sigma_k)\setminus(\sigma_i \cup \sigma_k) \subset \sigma_j\setminus\sigma_i$ for some $i,j \in [k-1]$, then all missing face containing both $u$ and $w$ have a vertex not in $\st_X(\sigma_i) \cup \st_X(\sigma_j)$, and also a vertex not in $\st_X(\sigma_i \cup \sigma_k) \cup \st_X(\sigma_j \cup \sigma_k)$ since $\st_X(\sigma_i \cup \sigma_k) \cup \st_X(\sigma_j \cup \sigma_k) \subset \st_X(\sigma_i) \cup \st_X(\sigma_j)$.
Thus, by the induction hypothesis, $X_1 \cap X_2$ is contractible.
Therefore,  Proposition~\ref{prop2} (c) implies that $X_1 \cup X_2=\bigcup_{i=1}^k \st(\sigma_i)$ is contractible.
\end{proof}

Now we are ready to give a proof of Theorem~\ref{thm:main}.
\begin{proof}[Proof of Theorem~\ref{thm:main}]
Let $K=\bigcup_{i=1}^k \st(\tilde{e}_i)$. First we show $I(H)=\st(v) \cup K$.
We obviously have $\st(v) \cup K \subset I(H)$, and hence it is sufficient to show that $I(H) \subset \st(v) \cup K$. Take $\sigma \in I(H)$.
If $\sigma \cup \{v\} \in I(H)$, then $\sigma \in \st(v)$.
Otherwise, if $\sigma \cup \{v\} \not\in I(H)$, then $\sigma \cup \{v\}$ contains an edge $e$ of $H$.
Since $\sigma \in I(H)$, the edge $e$ must contain the vertex $v$.
This implies $e = e_j$ for some $j \in [k]$, and it follows that $\tilde{e}_j \subset \sigma$.
Thus $\sigma \cup \tilde{e}_j=\sigma \in \st(\tilde{e}_j) \subset K$.
In both cases, we have $I(H) \subset \st(v) \cup K$.

Now, if we show that both $\st(v)$ and $K$ are contractible, then it implies $I(H) \simeq \Sigma(\st(v) \cap K)$ by Proposition~\ref{prop2} (c).
Clearly $\st(v)$ is contractible, so it is sufficient to show that $K$ is contractible.
This follows from an immediate application of Lemma~\ref{lem1} to $K=\bigcup_{i=1}^k \st(\tilde{e}_i)$.
It remains to check the the following two conditions to apply Lemma~\ref{lem1}:
\begin{enumerate}
    \item $\bigcup_{i=1}^k \tilde{e}_i \in I(H)$.
    \item If $u \in \tilde{e}_i\setminus\tilde{e}_j$ and $w \in \tilde{e}_j\setminus\tilde{e}_i$, then all missing face containing both $u$ and $w$ have a vertex not in $\st(\tilde{e}_i) \cup \st(\tilde{e}_j)$.
\end{enumerate}

To prove (1), assume $\bigcup_{i=1}^k \tilde{e}_i \not\in I(H)$.
Then there is $e \in H$ such that $e \subset \bigcup_{i=1}^k \tilde{e}_i \subset V(H)\setminus\{v\}$.
Recall that every edge in $H$ is inclusion-minimal, so $e \not\subset \tilde{e}_i$ for each $i \in [k]$.
Now we claim that there is an induced Berge cycle of length $3$ containing $v$, which contradicts the assumption.
Since $e \not\subset \tilde{e}_i$ for each $i \in [k]$ and $e \subset \bigcup_{i=1}^{k} \tilde{e}_i$, we can take two distinct index $i_1,i_2 \in [k]$ such that $e \cap (\tilde{e}_{i_1} \setminus \tilde{e}_{i_2}) \neq \varnothing$ and $e \cap (\tilde{e}_{i_2} \setminus \tilde{e}_{i_1}) \neq \varnothing$.
Take $u_1 \in e \cap (\tilde{e}_{i_1} \setminus \tilde{e}_{i_2})$ and $u_2 \in e \cap (\tilde{e}_{i_2} \setminus \tilde{e}_{i_1})$.
Note that $e \cap \{v,u_1,u_2\}=\{u_1,u_2\}$, ${e}_{i_1} \cap \{v,u_1,u_2\}=\{v,u_1\}$, and ${e}_{i_2} \cap \{v,u_1,u_2\}=\{v,u_2\}$.
Hence, $v~e_{i_1}~u_1~e~u_2~e_{i_2}$ is an induced Berge cycle of length $3$ containing $v$.
Therefore, we have $\bigcup_{i=1}^k \tilde{e}_i \not\in I(H)$.

Now we prove (2).
First, note that the missing faces of $I(H)$ are exactly the edges of $H$.
Suppose the contrary that there is a missing face $e$ of $I(H)$, i.e. an edge $e$ of $H$, containing both $u$ and $w$ such that all vertices in $e$ are contained in $\st(\tilde{e}_i) \cup \st(\tilde{e}_j)$.
Note that $v$ is not a vertex of $\st(\tilde{e}_i) \cup \st(\tilde{e}_j)$, hence we have $v \not\in e$.
Similarly as in the proof of (1), $v~e_{i}~u~e~w~e_{j}$ is an induced Berge cycle of length $3$ containing $v$.
This is a contradiction, thus shows (2).
\end{proof}

Now we recall Theorem~\ref{thm:cor}, which can be obtained from Theorem~\ref{thm:main}.
\begingroup
\def\thetheorem{\ref{thm:cor}}
\begin{theorem}
Let $H$ be a hypergraph and $v$ be a vertex of $H$ such that $v$ is not isolated and is not contained in an induced Berge cycle of length $3$.
\begin{enumerate}
    \item There is a hypergraph $H'$ with $V(H') \subset V(H)\setminus\{v\}$ such that $I(H) \simeq \Sigma I(H')$.
    \item $I(H)$ has the homotopy type of a suspension.
    \item $\tilde{H}_1(I(H))$ is a free abelian group.
\end{enumerate}
\end{theorem}
\addtocounter{theorem}{-1}
\endgroup
\begin{proof}
Let $e_1,e_2,\ldots,e_k$ be the edges of $H$ containing $v$ and $\tilde{e}_i=e_i \setminus \{v\}$.
By Theorem~\ref{thm:main}, we know $I(H) \simeq \Sigma (\st(v) \cap (\bigcup_{i=1}^{k} \st(\tilde{e}_i)))$.
Let $K=\st(v) \cap (\bigcup_{i=1}^{k} \st(\tilde{e}_i))$.
To prove (1), it is enough to see that $K$ is a simplicial complex defined on a subset of $V(H) \setminus \{v\}$.
This is obvious since each $\st(\tilde{e}_i)$ is a complex on $V(H) \setminus \{v\}$.
For (2), it immediately follows from (1).
For (3), since $I(H) \simeq \Sigma K$, we have $\tilde{H}_1(I(H)) \cong \tilde{H}_0(K)$ by Proposition~\ref{prop1}.
$\tilde{H}_0(K)$ is always a free abelian group, thus it implies that $\tilde{H}_1(I(H))$ is also a free abelian group.
\end{proof}

\section{Structural Condition for the Star cluster Theorem}\label{sec:structure}
Let $H$ be a hypergraph on $V$ and $v$ be a vertex of $H$ that is not isolated and not contained in an induced Berge cycle of length $3$.
Let $e_1,e_2,\ldots,e_k$ be the edges of $H$ containing $v$ and $\tilde{e}_i=e_i\setminus\{v\}$.
We define a new hypergraph, the \textit{star dissolution} $H_v$ at $v$ on $H$ as the set of missing faces of $st(v) \cap (\bigcup_{i=1}^k \st(\tilde{e}_i))$.
Then, by Theorem~\ref{thm:main}, we have $I(H) \simeq \Sigma I(H_v)$.
In this section, we provide a characterization of the edges of $H_v$.
It will explain why we call the hypergraph $H_v$ as ``star dissolution".
\begin{theorem}\label{thm:structure}
Let $H$ and $H_v$ be defined as above.
Then we have $\tilde{e}_1,\tilde{e}_2,\ldots,\tilde{e}_k \in H_v$.
Furthermore, $f \in H_v \setminus \{\tilde{e}_1,\tilde{e}_2,\ldots,\tilde{e}_k\}$ if and only if $\tilde{e}_i \not\subset f$ for all $i \in [k]$ and the following holds.
\begin{enumerate}[(i)]
    \item For each $i \in [k]$, there is an edge $h_i \in H \setminus \{e_1,e_2,\ldots,e_k\}$ such that $h_i \subset f \cup e_i$, and
    \item $f$ is inclusion-minimal subject to (i).
\end{enumerate}
\end{theorem}
The proof will be given in Section~\ref{sec:4.1}.
Using the above characterization of the edges of $H_v$, we describe how to construct $H_v$ from the original hypergraph $H$.

\begin{corollary}\label{cor:structure}
$H_v$ can be obtained by the following process.
\begin{enumerate}
    \item Remove $v$ from each $e_i$, that is, instead of $e_i$ add $\tilde{e}_i=e_i\setminus\{v\}$ as a new element.
    \item All edges not containing $v$ remain the same.
    \item Take $f_i \in H \setminus \{e_1,e_2,\ldots,e_k\}$ such that $f_i \cap e_i \neq \varnothing$ for each $i \in [k]$. Let $f=\bigcup_{i=1}^{k} (f_i\setminus e_i)$.
    We add $f$ as a new element. 
    \item Let $H_v$ be the set of all inclusion-minimal elements among all elements obtained by (1)--(3).
\end{enumerate}
\end{corollary}

For instance, in Figure~\ref{fig:sc_ex1}, we obtain more than $1$ new element at the step (3) and we may ignore singleton edges.
In Figure~\ref{fig:sc_ex2}, we apply the above process for a $3$-uniform hypergraph.

\begin{figure}[htbp]
    \centering
    \includegraphics[scale=1.6]{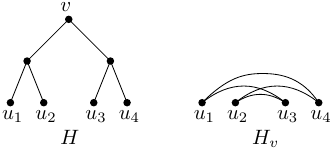}
    \caption{By applying Theorem~\ref{thm:structure} to $H$ and $v$, we obtain that $H_v$ is isomorphic to $K_{2,2}$.}
    \label{fig:sc_ex1}
\end{figure}

\begin{figure}[htbp]
    \centering
    \includegraphics[scale=1.4]{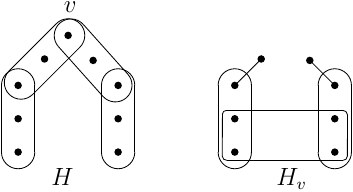}
    \caption{We obtain $H_v$ by applying Theorem~\ref{thm:structure} for a $3$-uniform hypergraph $H$ and $v$.}
    \label{fig:sc_ex2}
\end{figure}

\begin{proof}[Proof of Corollary~\ref{cor:structure}]
We first need to check that all edges of $H_v$ appear during the process (1)--(3).
Obviously, $\tilde{e}_1,\ldots,\tilde{e}_k \in H_v$ appear in the step (1).
By Theorem~\ref{thm:structure}, all other edges in $f \in H_v \setminus \{\tilde{e}_1,\ldots,\tilde{e}_k\}$ satisfies that
\begin{enumerate}[(i)]
    \item For each $i \in [k]$, there is an edge $h_i \in H \setminus \{e_1,e_2,\ldots,e_k\}$ such that $h_i \subset f \cup e_i$, and
    \item $f$ is inclusion-minimal subject to (i).
\end{enumerate}
Suppose $h_i \cap e_i = \varnothing$ for some $i \in [k]$.
Then $h_i \subset f \cup e_i$ implies $h_i \subset f$.
In addition, since $h_i \subset h_i \cup e_j$ for all $j \in [k]$, we have $h_i=f$ by the minimality of $f$.
Thus, in this case, we obtain $f =h_i \in H \setminus \{e_1,e_2,\ldots,e_k\}$.
Hence, $f$ appears in the step (2).

Thus assume $h_i \cap e_i \neq \varnothing$ for all $i \in [k]$.
Then define $h=\bigcup_{i=1}^{k}(h_i \setminus e_i)$.
For each $i \in [k]$, we have $h_i \setminus e_i \subset f$ since $h_i \subset f \cup e_i$, hence it implies $h \subset f$.
On the other hand, since $h_i \subset h \cup e_i$ for all $i \in [k]$, we obtain $h=f$ by the minimality of $f$.
Thus, $f=h=\bigcup_{i=1}^{k}(h_i \setminus e_i)$ appears in the step (3).

Now it remains to check that all elements appearing during the process (1)--(3) contain some edge of $H_v$.
Let $e$ be an element appearing during the process (1)--(3).
\begin{itemize}
    \item If $e$ appears in the step (1), then $e=\tilde{e}_i$ for some $i$ and we know $e=\tilde{e}_i \in H_v$ by Theorem~\ref{thm:structure}.
    \item If $e$ appears in the step (2), then $e \in H\setminus\{e_1,\ldots,e_k\}$.
    In this case, we have $\tilde{e}_i \not\subset e$ and $e$ satisfies the condition (i) of Theorem~\ref{thm:structure} since $e \subset e \cup e_i$ for all $i \in [k]$.
    Thus by Theorem~\ref{thm:structure}, $e$ contains some edge of $H_v$.
    \item If $e$ appears in the step (3), then $e=\bigcup_{i=1}^{k}(f_i \setminus e_i)$ for some $f_i \in H\setminus\{e_1,\ldots,e_k\}$ such that $f_i \cap e_i \neq \varnothing$.
    Then $e$ satisfies the condition (i) of Theorem~\ref{thm:structure} since $f_i \subset e \cup e_i$ for each $i \in [k]$, and hence $e$ contains some edge of $H_v$ by Theorem~\ref{thm:structure}.
\end{itemize}
This completes the proof.
\end{proof}

Based on Corollary~\ref{cor:structure}, Theorem~\ref{thm:main} can be applied to compute the homotopy type of the independence complex of hypergraphs.
Here, we introduce one interesting application of Theorem~\ref{thm:main}.

Let $P_{n,k}$ be a \textit{$k$-uniform tight path} on $n$ vertices, that is, the edges of $P_{n,k}$ are the set of $k$ consecutive vertices of the vertices $v_1,v_2,\ldots,v_n$ of $P_{n,k}$.
Let $H$ be a hypergraph containing $P_{2k-2,k}$ defined on $a_1,a_2,\ldots,a_{2k-2}$ as a subgraph.
Let $\mathcal{L}_k(H)$ be the hypergraph obtained from $H$ by deleting all edges of $P_{2k-2,k}$ on $a_1,a_2,\ldots,a_{2k-2}$ and adding new vertices $b_1,b_2,\ldots,b_{k+1}$ and new edges of $P_{3k-1,k}$ on $a_1,\ldots,a_{k-1},b_1,\ldots,b_{k+1},a_k,\ldots,a_{2k-2}$.
Note that $H$ can have edges on $a_1,a_2,\ldots,a_{2k-2}$ incomparable with any edge of $P_{2k-2,k}$.
For instance, see Figure~\ref{fig:tightpath} when $k=3$.

\medskip
\begin{figure}[htbp]
    \centering
    \includegraphics[scale=1.4]{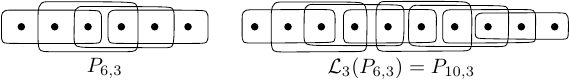}
    \caption{For a $3$-uniform tight path $P_{6,3}$, we obtain $\mathcal{L}_3(P_{6,3})=P_{10,3}$.}
    \label{fig:tightpath}
\end{figure}

If $H$ is a graph, then $k=2$ and $P_{2k-2,k}=P_{2,2}$ is a path of length $1$ on $a_1,a_2$ and $\mathcal{L}_2(H)$ is obtained from $H$ by replacing an edge $a_1a_2$ by a path of length $4$.
As mentioned in the introduction, Csorba~\cite{Csorba} showed that $I(\mathcal{L}_2(H)) \simeq \Sigma I(H)$.
The following theorem generalizes Csorba's result.
\begin{theorem}\label{thm:tightpath}
Let $H$ and $\mathcal{L}_k(H)$ be defined as above.
Then $I(\mathcal{L}_k(H)) \simeq \Sigma^{k-1}I(H)$.
\end{theorem}

\begin{proof}
We prove by applying the star cluster theorem to $\mathcal{L}_k(H)$ for vertices $b_k,b_{k-1},\ldots,b_2$ in order.
Let $F_1=\mathcal{L}_k(H)_{b_k}$ and $F_i=(F_{i-1})_{b_{k-i+1}}$ for $2 \leq i \leq k-1$.
Then $F_i$ is the hypergraph obtained by applying the star cluster theorem to $\mathcal{L}_k(H)$ for vertices $b_k,b_{k-1},\ldots,b_{k-i+1}$ in order.

\begin{claim}\label{claim3}
For each $i \in [k-1]$, the edges of $F_i$ induced by $a_1,a_2,\ldots,a_{2k-2},b_1,b_2,\ldots,b_{k-i},b_{k+1}$ are given as follows:
\begin{enumerate}
    \item[a)] All edges of $P_{2k-2i,k-i}$ on $b_1,b_2,\ldots,b_{k-i},b_{k+1},a_k,\ldots,a_{2k-i-2}$ are edges of $F_i$.
    \item[b)] $\{a_j,a_{j+1}\ldots,a_{k-1},b_1,b_2,\ldots,b_j\} \in F_i$ for $1 \leq j \leq k-i-1$.
    \item[c)] $\{a_{\ell},\ldots,a_{k-1},b_1,\ldots,b_m,a_{k+m},\ldots,a_{k+\ell-1}\}\in F_i$ for $k-i \leq \ell \leq k-1$, $0 \leq m \leq k-i-1$. Note that we obtain $\{a_{\ell},a_{\ell+1},\ldots,a_{k+\ell-1}\}$ when $m=0$.
    \item[d)] All other edges on $a_1,a_2,\ldots,a_{2k-2}$ in $H$ incomparable with any edge of $P_{2k-2,k}$.
\end{enumerate}
\end{claim}
\begin{proof}[Proof of Claim~\ref{claim3}]
We prove by the induction on $i$ using Theorem~\ref{thm:structure} and Corollary~\ref{cor:structure}.
For $i=1$, the edges of $F_1$ induced by $a_1,a_2,\ldots,a_{2k-2},b_1,b_2,\ldots,b_{k-1},b_{k+1}$ are given as follows: 
We first consider the edges $e_1,\ldots,e_k$ of $\mathcal{L}_k(H)$ containing $b_k$. Let $e_1=\{b_1,\ldots,b_k\}$, $e_2=\{b_2,\ldots,b_{k+1}\}$, and $e_j=\{b_j,\ldots,b_{k+1},a_k,\ldots,a_{k+j-3}\}$ for $3 \leq j \leq k$.
\begin{itemize}
    \item In the step (1), we obtain $\tilde{e}_1,\ldots,\tilde{e}_k$, i.e. the edges of $P_{2k-2,k-1}$ on $b_1,\ldots,b_{k-1},b_{k+1},a_k,\ldots,a_{2k-3}$.
    \item In the step (2), we obtain $\{a_j,\ldots,a_{k-1},b_1,\ldots,b_j\}$ for $1 \leq j \leq k-1$ and all other edges on $a_1,a_2,\ldots,a_{2k-2}$ in $H$ incomparable with any edge of $P_{2k-2,k}$.
    \item In the step (3), we obtain $\{a_{k-1},b_1,\ldots,b_m,a_{k+m},\ldots,a_{2k-2}\}$ for $0 \leq m \leq k-2$ by taking $f_j=\{a_{k-1},b_1,\ldots,b_k\}$ for $1 \leq j \leq m+1$ and $f_j=\{b_{k+1},a_k,\ldots,a_{2k-2}\}$ for $m+2 \leq j \leq k$. 
\end{itemize} 
Note that $\{a_{k-1},b_1,\ldots,b_{k-1}\}$ is not minimal since we obtained $\{b_1,\ldots,b_{k-1}\}$ in the step (1), and all other elements obtained from (1)--(3) is minimal, hence we obtain that the statement holds when $i=1$.

Now suppose $i>1$ and $F_{i-1}$ satisfies the statement, that is, the edges of $F_{i-1}$ induced by \\
$a_1,a_2,\ldots,a_{2k-2},b_1,b_2,\ldots,b_{k-i+1},b_{k+1}$ are given as follows:
\begin{enumerate}
    \item[a)] All edges of $P_{2k-2i+2,k-i+1}$ on $b_1,b_2,\ldots,b_{k-i+1},b_{k+1},a_k,\ldots,a_{2k-i-1}$ are edges of $F_{i-1}$.
    \item[b)] $\{a_j,a_{j+1}\ldots,a_{k-1},b_1,b_2,\ldots,b_j\} \in F_{i-1}$ for $1 \leq j \leq k-i$.
    \item[c)] $\{a_{\ell},\ldots,a_{k-1},b_1,\ldots,b_m,a_{k+m},\ldots,a_{k+\ell-1}\}\in F_{i-1}$ for $k-i+1 \leq \ell \leq k-1$, $0 \leq m \leq k-i$.
    \item[d)] All other edges on $a_1,a_2,\ldots,a_{2k-2}$ in $H$ incomparable with any edge of $P_{2k-2,k}$.
\end{enumerate}
By Theorem~\ref{thm:structure} and Corollary~\ref{cor:structure}, we obtain the edges of $F_i=(F_{i-1})_{b_{k-i+1}}$ induced by $a_1,a_2,\ldots,a_{2k-2}$, $b_1,b_2,\ldots,b_{k-i},b_{k+1}$ are given as follows: 
Consider the edges $e_1,\ldots,e_{k-i+1}$ of $F_{i-1}$ containing $b_{k-i+1}$.
Let $e_1=\{b_1,\ldots,b_{k-i+1}\}$, $e_2=\{b_2,\ldots,b_{k-i+1},b_{k+1}\}$ and $e_j=\{b_j,\ldots,b_{k-i+1},b_{k+1},a_k,\ldots,a_{k+j-3}\}$ for $3 \leq j \leq k-i+1$.
\begin{itemize}
    \item In the step (1), we obtain $\tilde{e}_1,\ldots,\tilde{e}_{k-i+1}$, i.e. the edges of $P_{2k-2i,k-i}$ on $b_1,b_2,\ldots,b_{k-i},b_{k+1}$, $a_k,\ldots,a_{2k-i-2}$.
    \item In the step (2), we obtain $\{a_j,\ldots,a_{k-1},b_1,\ldots,b_j\}$ for $1 \leq j \leq k-i$, \\
    $\{a_{\ell},\ldots,a_{k-1},b_1,\ldots,b_m,a_{k+m},\ldots,a_{k+\ell-1}\}$ for $k-i+1 \leq \ell \leq k-1$, $0 \leq m \leq k-i$, and\\
    all other edges on $a_1,a_2,\ldots,a_{2k-2}$ in $H$ incomparable with any edge of $P_{2k-2,k}$.
    \item In the step (3), we obtain $\{a_{k-i},\ldots,a_{k-1},b_1,\ldots,b_m,a_{k+m},\ldots,a_{2k-i-1}\}$ for $0 \leq m \leq k-i-1$ by taking $f_j=\{a_{k-i},\ldots,a_{k-1},b_1,\ldots,b_{k-i}\}$ for $1 \leq j \leq m+1$ and $f_j=\{b_{k+1},a_k,\ldots,a_{2k-i-1}\}$ for $m+2 \leq j \leq k-i+1$.
\end{itemize}
Note that $\{a_{\ell},\ldots,a_{k-1},b_1,\ldots,b_{k-i},a_{2k-i},\ldots,a_{k+\ell-1}\}$ for each $k-i \leq \ell \leq k-1$ is not minimal since we obtained $\{b_1,\ldots,b_{k-i}\}$ in the step (1), and all other elements obtained from (1)--(3) is minimal. 
Thus we can conclude that $F_i$ also satisfies the statement.
\renewcommand{\qedsymbol}{$\blacksquare$}
\end{proof}
By Claim~\ref{claim3}, we obtain that the edges of $F_{k-1}$ on $a_1,a_2,\ldots,a_{2k-2},b_1,b_{k+1}$ is given as follows:
\begin{enumerate}
    \item[a)] All edges of $P_{2,1}$ on $b_1,b_{k+1}$ are edges of $F_{k-1}$, that is, $\{b_1\},\{b_{k+1}\} \in F_{k-1}$.
    \item[c)] $\{a_{\ell},a_{\ell+1},\ldots,a_{k+\ell-1}\} \in F_{k-1}$ for $1 \leq \ell \leq k-1$. 
    \item[d)] All other edges on $a_1,a_2,\ldots,a_{2k-2}$ in $H$ incomparable with any edge of $P_{2k-2,k}$.
\end{enumerate}
Thus we obtain that $F_{k-1}=H \cup \{\{b_1\},\{b_{k+1}\}\}$.
Since $b_1,b_{k+1}$ are forming singleton edges, they are not appear in $I(F_{k-1})$.
Hence, we can conclude that $I(F_{k-1})=I(H)$.
Now, by Theorem~\ref{thm:main}, we obtain  $$I(\mathcal{L}_k) \simeq \Sigma I(F_1) \simeq \Sigma^{k-1} F_{k-1} \simeq \Sigma^{k-1} I(H)$$
as we wanted.
\end{proof}

\subsection{Proof of Theorem~\ref{thm:structure}}\label{sec:4.1}
For each $i \in [k]$, define the hypergraph $H_i$ on $V\setminus\{v\}$ as \[H_i=\{e \setminus e_i : e \in H\setminus\{e_1,e_2,\ldots,e_k\}\} \cup \{\tilde{e}_1,\tilde{e}_2,\ldots,\tilde{e}_k\}.\]
Note that edges in $H_i$ may not be inclusion-minimal.

\begin{lemma}\label{lem2}
Let $H$ and $H_i$ be defined as above. Then $I(H_i)=\st(v) \cap \st(\tilde{e}_i)$.
\end{lemma}
\begin{proof}
First we show $I(H_i) \subset \st(v) \cap \st(\tilde{e}_i)$.
Take $\sigma \in I(H_i)$.
Suppose $\sigma \cup \{v\} \not\in I(H)$.
Then there is $e \in H$ such that $e \subset \sigma \cup \{v\}$. 
If $e \neq e_j$ for all $j \in [k]$, then $e\setminus e_i \subset (\sigma \cup \{v\})\setminus e_i \subset \sigma$. Since $e\setminus e_i$ is an edge of $H_i$, this is a contradiction to $\sigma \in I(H_i)$.
If $e=e_j$ for some $j \in [k]$, then we have $\tilde{e}_j \subset \sigma$.
Similarly as above, we again obtain an edge $\tilde{e}_j$ of $H_i$ contained in $\sigma$, which is a contradiction to $\sigma \in I(H_i)$.
This implies $\sigma \cup \{v\} \in I(H)$, that is, $\sigma \in \st(v)$.
Next, suppose $\sigma \cup \tilde{e}_i \not\in I(H)$.
Then there is $e \in H$ such that $e \subset \sigma \cup \tilde{e}_i$.
Note that $e \neq e_j$ for all $j \in [k]$ since $\sigma, \tilde{e}_i \subset V\setminus \{v\}$.
Thus we have $e \setminus e_i \subset (\sigma\cup\tilde{e}_i)\setminus e_i \subset \sigma$.
Since $e \setminus e_i \in H_i$, it is a contradiction to $\sigma \in I(H_i)$, implying $\sigma \in \st(\tilde{e}_i)$.
Therefore, we have $I(H_i) \subset \st(v) \cap \st(\tilde{e}_i)$.

Now we prove $I(H_i) \supset \st(v) \cap \st(\tilde{e}_i)$.
Take $\sigma \in \st(v) \cap \st(\tilde{e}_i)$.
First, note that $\sigma$ does not contain $v$ since $v \notin \st(\tilde{e}_i)$.
Also, since $\sigma \in \st(v)$ and no $\tilde{e}_j$ is contained in $\st(v)$, we have $\tilde{e}_j \not\subset \sigma$ for all $j \in [k]$.
Now suppose $\sigma \notin I(H_i)$.
Then there is $e \in H_i$ such that $e \subset \sigma$.
Observing $e \neq \tilde{e}_j$ for all $j \in [k]$, it must be $e=f\setminus e_i$ for some $f \in H\setminus\{e_1,e_2,\ldots,e_k\}$.
Then $f \subset e\cup\tilde{e}_i \subset \sigma \cup \tilde{e}_i$, and this is a contradiction to $\sigma \in \st(\tilde{e}_i)$.
Thus, we have $\sigma \in I(H_i)$ and hence we conclude that $I(H_i) \supset \st(v) \cap \st(\tilde{e}_i)$.
This completes the proof.
\end{proof}

For a hypergraph $F$, let $\cl(F)$ be the \textit{closure} of $F$ and define it as
\[\cl(F)=\{e \subset V: f \subset e \text{ for some }f \in F\}.\]
Then we have $\cl(F)=2^{V} \setminus I(F)$ for any hypergraph $F$.
If all edges of $F$ are inclusion-minimal, then $f \in F$ if and only if $f$ is inclusion-minimal in $\cl(F)$.

\begin{lemma}\label{lem3}
Let $F,F_1,F_2,\ldots,F_m$ be hypergraphs on $V$ and $I(F)=\bigcup_{i=1}^{m} I(F_i)$. 
Assume that all edges in $F$ are inclusion-minimal.
Then $f \in F$ if and only if the following holds.
\begin{enumerate}
    \item There is an edge $e_i \subset f$ of $F_i$ for each $i \in [m]$, and
    \item $f$ is inclusion-minimal subject to (1).
    \end{enumerate}
\end{lemma}
\begin{proof}
First, we know that $I(F)=\bigcup_{i=1}^{m} I(F_i)$ implies 
\[\cl(F)=2^V \setminus I(F)=2^V \setminus \bigcup_{i=1}^{m} I(F_i) =\bigcap_{i=1}^{m} 2^V \setminus I(F_i) =\bigcap_{i=1}^{m} \cl(F_i).\] 

Thus, we obtain that $f \in F$ if and only if $f$ is inclusion-minimal in $\bigcap_{i=1}^{m}\cl(F_i)$.
Thus it is enough to show that every element $f$ in $\bigcap_{i=1}^{m}\cl(F_i)$ satisfies the condition that there is an edge $e_i \subset f$ of $F_i$ for each $i \in [m]$, and this directly follows from the definition of $\cl(F_i)$.
\end{proof}

Now we prove Theorem~\ref{thm:structure}.
\begingroup
\def\thetheorem{{\ref{thm:structure}}}
\begin{theorem}
We have $\tilde{e}_1,\tilde{e}_2,\ldots,\tilde{e}_k \in H_v$.
Furthermore, $f \in H_v \setminus \{\tilde{e}_1,\tilde{e}_2,\ldots,\tilde{e}_k\}$ if and only if $\tilde{e}_i \not\subset f$ for all $i \in [k]$ and the following holds.
\begin{enumerate}[(i)]
    \item For each $i \in [k]$, there is an edge $h_i \in H \setminus \{e_1,e_2,\ldots,e_k\}$ such that $h_i \subset f \cup e_i$, and
    \item $f$ is inclusion-minimal subject to (i).
\end{enumerate}
\end{theorem}
\addtocounter{theorem}{-1}
\endgroup
\begin{proof}
Let $H_i$ be defined as above.
Note that $I(H_v)=\bigcup_{i=1}^{k} I(H_i)$ by Theorem~\ref{thm:main} and Lemma~\ref{lem2}.
We will prove the statement by applying Lemma~\ref{lem3} to $I(H_v)=\bigcup_{i=1}^{k} I(H_i)$.
By Lemma~\ref{lem3}, $f \in H_v$ if and only if 
\begin{enumerate}
    \item there is an edge $f_i \subset f$ of $H_i$ for each $i \in [k]$, and
    \item $f$ is inclusion-minimal subject to (1).
\end{enumerate}

We first prove that $\tilde{e}_j \in H_v$ for each $j \in [k]$.
Since $\tilde{e}_j$ itself is an edge of $H_i$, thus it contains an edge of $H_i$ for all $i \in [k]$.
To show the minimality of $\tilde{e}_j$, suppose there is a proper subset $g$ of $\tilde{e}_j$ such that $g$ contains an edge $h_j$ of $H_j$.
If $h_j=\tilde{e}_i$ for some $i \in [k]$, then we have $\tilde{e}_i=h_j \subset g \subsetneq \tilde{e}_j$, and hence we obtain $e_i \subsetneq e_j$ and it contradicts the assumption that every edge in $H$ is inclusion-minimal.
Thus $h_j=e \setminus e_j$ for some $e \in H\setminus\{e_1,e_2,\ldots,e_k\}$.
Then we obtain $e \setminus e_j \subset g \subsetneq \tilde{e}_j$, and it implies that $e \subsetneq e_j$, which is again a contradiction to the assumption that every edge in $H$ is inclusion-minimal.
Thus we obtain that $\tilde{e}_1,\tilde{e}_2,\ldots,\tilde{e}_k \in H_v$.

Now to prove the remaining part, we claim the following.
\begin{claim}\label{claim2}
    Suppose $f \subset V(H) \setminus\{v\}$ and $\tilde{e}_j \not\subset f$ for all $j \in [k]$.
    Then there is an edge $f_i\subset f$ of $H_i$ if and only if there is an edge $h_i \in H \setminus \{e_1,e_2,\ldots,e_k\}$ such that $h_i \subset f \cup e_i$.
\end{claim}
\begin{proof} 
Suppose that there is an edge $f_i\subset f$ of $H_i$.
Since $\tilde{e}_j \not\subset f$ for all $j \in [k]$, we know $f_i\neq \tilde{e}_j$ for all $j \in [k]$.
Thus there is an edge $h_i \in H\setminus\{e_1,e_2,\ldots,e_k\}$ such that $f_i=h_i \setminus e_i$, and hence $h_i \subset f \cup e_i$.
Now for the opposite, if there is an edge $h_i \in H \setminus \{e_1,e_2,\ldots,e_k\}$ such that $h_i \subset f \cup e_i$, then obviously $h_i \setminus e_i \in H_i$ and $h_i \setminus e_i \subset f \setminus e_i \subset f$.
\renewcommand{\qedsymbol}{$\blacksquare$}
\end{proof}

To prove the only if part, assume that $f \in H_v \setminus\{\tilde{e}_1,\tilde{e}_2,\ldots,\tilde{e}_k\}$.
Then by the minimality of each edge of $H_v$, we have $\tilde{e}_i \not\subset f$ for all $i \in [k]$.
In addition, by Lemma~\ref{lem3} and Claim~\ref{claim2}, it is obvious that $f$ satisfies (i),(ii).

Now, to prove the if part, assume that $\tilde{e}_i \not\subset f$ for all $i \in [k]$ and $f$ satisfies (i),(ii).
Then it is obvious that $f \not\in \{\tilde{e}_1,\tilde{e}_2,\ldots,\tilde{e}_k\}$.
In addition, by Claim~\ref{claim2}, the condition (i) is equivalent to that there is an edge $f_i \subset f$ of $H_i$ for each $i \in [k]$.
Then by Lemma~\ref{lem3}, we obtain $f \in H_v$, and hence we have $f \in H\setminus\{\tilde{e}_1,\tilde{e}_2,\ldots,\tilde{e}_k\}$.
This completes the proof.
\end{proof}

\section{Bounds on the total Betti numbers}\label{sec:KM}

In this section, we present the proof of Theorem~\ref{thm:main2}.
Moreover, we show that if a hypergraph $H$ has only bounded number of disjoint ternary Berge cycles, then $\beta(I(H))$ is also bounded.

\subsection{Proof of Theorem~\ref{thm:main2}}
Let $H$ be a hypergraph and $e=\{v_1,v_2,\ldots,v_k\}$ be an edge of $H$.
We define a new hypergraph, the \textit{edge-star expansion} $H_e$ at $e$ on $H$ obtained from $H$ by deleting the edge $e$ and adding new vertices $\{w,u_1,u_2,\ldots,u_k\}$ and new edges $$\{\{w,u_1\},\{w,u_2\},\ldots,\{w,u_k\},\{u_1,v_1\},\{u_2,v_2\},\ldots,\{u_k,v_k\}\}.$$
That is, $$H_e=(H\setminus \{e\}) \cup \{\{w,u_1\},\{w,u_2\},\ldots,\{w,u_k\},\{u_1,v_1\},\{u_2,v_2\},\ldots,\{u_k,v_k\}\}.$$
See Figure~\ref{fig:sc_ex3} for an example with $k=4$.
\begin{figure}[htbp]
    \centering
    \includegraphics[scale=1.4]{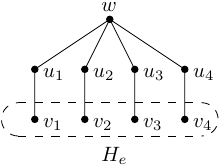}
    \caption{The edge-star expansion $H_e$ at $e=\{v_1,v_2,v_3,v_4\}$.}
    \label{fig:sc_ex3}
\end{figure}

\begin{lemma}\label{lem5}
Let $H$ be a hypergraph and $e$ be an edge of $H$.
Then $I(H_e) \simeq \Sigma I(H)$.
\end{lemma}
\begin{proof}
We claim that $(H_e)_w=H\cup\{\{u_1\},\{u_2\},\ldots,\{u_k\}\}$.
Note that $I((H_e)_w)=I(H)$ since all vertices in the singleton edges cannot appear in the independence complex.
Then we can conclude that $I(H_e) \simeq \Sigma I((H_e)_w) = \Sigma I(H)$ by Theorem~\ref{thm:main}.

Thus it is enough to show $(H_e)_w=H\cup\{\{u_1\},\{u_2\},\ldots,\{u_k\}\}$, and for that, we will use Corollary~\ref{cor:structure}.
By Corollary~\ref{cor:structure}, we obtain the edges of $(H_e)_w$ as follows:
\begin{itemize}
    \item In the step (1), we obtain $\{u_1\},\{u_2\},\ldots,\{u_k\}$ since $\{w,u_1\},\{w,u_2\},\ldots,\{w,u_k\}$ are the edges containing $w$ in $H_e$.
    \item In the step (2), we obtain all edges of $H_e$ not containing $w$, that means we obtain all edges of $H$ except $e$ and $\{u_1,v_1\},\{u_2,v_2\},\ldots,\{u_k,v_k\}$.
    \item In the step (3), since $\{u_i,v_i\}$ is the only edge in $H_e \setminus \{\{w,u_1\},\{w,u_2\},\ldots,\{w,u_k\}\}$ intersects with $\{w,u_i\}$, we only obtain $e=\bigcup_{i=1}^{k} \{v_i\}$. 
\end{itemize}
Now we collect only inclusion-minimal elements among obtained by (1)--(3), then we obtain $(H_e)_w=H\cup\{\{u_1\},\{u_2\},\ldots,\{u_k\}\}$ as we wanted.
\end{proof}

\begin{lemma}\label{lem6}
For a positive integer $k$, if $t(H) < k$, then $t(H_e) < k$.
\end{lemma}
\begin{proof}
Suppose $t(H_e) \geq k$, that is, $H_e$ has $k$ pairwise disjoint ternary Berge cycles $C_1,C_2,\ldots,C_{k}$.
If all edges in $C_1,C_2,\ldots,C_{k}$ are edges of $H$, then $t(H) \geq k$ which is a contradiction to the assumption.
Thus we may assume that $C_1=x_1e_1x_2e_2\cdots x_{3m}e_{3m}$ contains an edge in $H_e \setminus H$.
Recall that $$H_e \setminus H=\{\{w,u_1\},\{w,u_2\},\ldots,\{w,u_k\},\{u_1,v_1\},\{u_2,v_2\},\ldots,\{u_k,v_k\}\}.$$
Thus if $C_1$ has an edge in $H_e \setminus H$, then there are exactly $4$ consecutive edges in $C$ contained in $H_e \setminus H$. 
Without loss of generality, assume $e_1=\{u_1,v_1\}$, $e_2=\{w,u_1\}$, $e_3=\{w,u_2\}$, $e_4=\{u_2,v_2\}$ and $x_1=v_1$, $x_2=u_1$, $x_3=w$, $x_4=u_2$, $x_5=v_2$.
Note that all other edges $e_5, \ldots, e_{3m}$ are edges of $H$.
Then $C'_1=v_1 e v_2 e_5 x_6 e_6 \cdots x_{3m} e_{3m}$ is a Berge cycle of length $3m-3$ in $H$.

Now, for each $i \in \{2,3,\ldots,k\}$, the vertex $w$ does not appear in $C_i$ since $C_i$ is disjoint with $C_1$.
Then the above argument implies that $C_i$ cannot contain an edge in $H_e\setminus H$.
Thus we can easily check that $C'_1,C_2,\ldots,C_{k}$ are pairwise disjoint ternary Berge cycles in $H$, implying $t(H) \geq k$ which is again a contradiction.
\end{proof}

When $k=0$, the above lemma states the following:
\begin{corollary}\label{cor1}
Let $H$ be a hypergraph and $e$ be an edge of $H$.
If $H$ has no ternary Berge cycle, then neither does $H_e$.
\end{corollary}

Now we are ready to prove Theorem~\ref{thm:main2}. This can be done by converting the original hypergraph without a ternary Berge cycle to a graph without a cycle of length $0$ modulo $3$, preserving the total Betti number. Then it follows from the result in \cite{WZ25}:
\begin{theorem}\label{thm:graphcase}\cite{WZ25}
If a graph $G$ has no induced cycle of length $0$ modulo $3$, then $\beta(I(G)) \leq 1$.
\end{theorem}
\begingroup
\def\thetheorem{\ref{thm:main2}}
\begin{theorem}
If a hypergraph $H$ has no ternary Berge cycle, then $\beta(I(H)) \leq 1$.
\end{theorem}
\addtocounter{theorem}{-1}
\endgroup

\begin{proof}
Let $H$ be a hypergraph with no ternary Berge cycle.
We claim that there is a graph $G$ with no ternary cycle with $\beta(I(G))=\beta(I(H))$.
Then by Theorem~\ref{thm:graphcase}, it implies that $\beta(I(H))=\beta(I(G)) \leq 1$.

We apply induction on the number of edges of size at least $3$.
As mentioned in Section~\ref{sec:pre}, we assume that $H$ has no edge of size $1$.
Thus the base case of the induction is when all edge of $H$ has the size $2$.
In this case, the statement follows from Theorem~\ref{thm:graphcase}.
Now suppose $H$ has an edge $e$ of size at least $3$.
We observe the following about $H_e$:
\begin{itemize}
    \item By the definition, the number of edges of size at least $3$ in $H_e$ is one less than that of $H$.
    \item $\beta(I(H))=\beta(I(H_e))$ by Lemma~\ref{lem5} and Prop~\ref{prop1}.
    \item Corollary~\ref{cor1} implies that $H_e$ also has no ternary Berge cycle.
\end{itemize}
By the induction hypothesis, $\beta(I(H_e)) \leq 1$, thus it follows that $\beta(I(H)) \leq 1$.
\end{proof}

    A natural question that arises from Theorem~\ref{thm:main2} is which homotopy types that $I(H)$ can have for such hypergraphs.
    We suspect that, if $H$ contains no ternary Berge cycles, then $I(H)$ should be either contractible or homotopy equivalent to a sphere, as in the case of graphs~\cite{Kim22}.
    However, this does not immediately follow from our argument because of the existence of a complex whose suspension is homotopy equivalent to a sphere while the complex itself is not.
    By the double suspension theorem \cite{Cannon79}, the double suspension $\Sigma^2 X$ of a homology $d$-sphere $X$ is homotopy equivalent to $\mathbb{S}^{d+2}$.
    For instance, if we consider the Poincar\'{e} homology $3$-sphere, then its fundamental group has order $120$ and hence it is not homotopy equivalent to a sphere, but its double suspension is homotopy equivalent to $\mathbb{S}^5$.
    As a conclusion, the following question still remains:
\begin{question}\label{que1}
If a hypegraph $H$ has no ternary Berge cycle, then is $I(H)$ either contractible or homotopy equivalent to a sphere?
\end{question}

\begin{remark}
In the proof of Theorem~\ref{thm:main2}, one of the key parts is Corollary~\ref{cor1}, which says that if $H$ has no ternary Berge cycle, then neither does the edge-star expansion $H_e$ at $e$ for any edge $e$ of $H$.
It might be wondered whether the converse direction holds.
That is, if $H$ has no ternary Berge cycle, then does the star dissolution $H_v$ at $v$ also have no ternary Berge cycle for any vertex $v$ of $H$?
Unfortunately, the answer is negative in general, and this is the reason why we consider $H_e$ instead of $H_v$ to prove Theorem~\ref{thm:main2}.
See Figure~\ref{fig:sc_ex4} for an example.

\begin{figure}[htbp]
    \centering
    \includegraphics[scale=1.6]{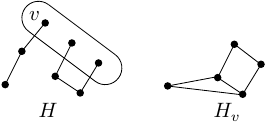}
    \caption{Although $H$ has no ternary Berge cycle, $H_v$ contains a triangle.}
    \label{fig:sc_ex4}
\end{figure}
\end{remark}

\subsection{Bounded number of ternary Berge cycles}
Now we investigate more general case when a hypergraph $H$ contains at most $k$ disjoint ternary Berge cycles, i.e. $t(H) \leq k$. We prove that $\beta(I(H))$ is bounded by a function of $k$. For this, we use the following Erd\H{o}s--P\'{o}sa theorem~\cite{Thom88} for cycles of length $0$ modulo $m$.

\begin{theorem}\cite{Thom88}\label{Erdos-Posa}
There is a function $g_m:\mathbb{N} \to \mathbb{N}$ satisfying the following.
If a graph $G$ contains at most $k$ vertex-disjoint cycles of length $0$ modulo $m$, then there are at most $g_m(k)$ vertices hitting all cycles of length $0$ modulo $m$.
\end{theorem}
The current best known bound for $g_m(k)$ is $g_m(k)=O_m(k \log k)$ by \cite{BHJR19}.

\begin{theorem}\label{thm:graphcase-general}
For every graph $G$ with $t(G) \leq k$, we have $\beta(I(G)) \leq 2^{g_3(k)}=2^{O(k \log k)}$.
\end{theorem}
\begin{proof}
For every vertex $v$ of $G$, we have $$I(G)=I(G-v) \cup I(G-N_G(v)) \text{ and } I(G-v) \cap I(G-N_G(v))=I(G-N_G[v]).$$
Then we have the Mayer--Vietoris sequence~\eqref{mvseq} for $(I(G),I(G-v),I(G-N_G(v))$.
\begin{equation*}
    \cdots \to \tilde{H}_i (I(G-N_G[v]) \to \tilde{H}_i(I(G-v)) \oplus \tilde{H}_i(I(G-N_G(v)) \xrightarrow{\gamma_i} \tilde{H}_i(I(G)) \xrightarrow{\delta_i} \tilde{H}_{i-1} (I(G-N_G[v]) \to \cdots.
\end{equation*}

Since $v$ is an isolated vertex of $G-N_G(v)$, we note that $I(G-N_G(v))$ is a cone with apex $v$, thus it is contractible.
This implies $\tilde{\beta}_i(I(G-N_G(v)))=0$ for all $i$.
Then by Lemma~\ref{mvseq-betti}, we obtain $$\tilde{\beta}_i(I(G)) \leq \tilde{\beta}_i(I(G-v))+\tilde{\beta}_{i-1}(I(G-N_G[v]))$$
for all $i$.
Thus we have $$\beta(I(G)) \leq \beta(I(G-v))+\beta(I(G-N_G[v])$$ for any vertex $v$.

Now, for disjoint vertex subsets $X$ and $Y$, let $G(X|Y)$ be the graph obtained from $G$ by deleting all vertices in $N_G[X] \cup Y$, that is, $G(X|Y)=G-N_G[X]-Y$.
For convenience, we omit curly brackets in the notation $G(X|Y)$, for instance we write $G(v|Y)$ instead of $G(\{v\}|Y)$.
Then the above inequality can be written as
\begin{align}\label{eq:vtx}
\beta(I(G)) \leq \beta(I(G(v|\varnothing)))+\beta(I(G(\varnothing|v)))    
\end{align}
for any vertex $v$. 

Given a vertex subset $U=\{u_1,u_2,\ldots,u_m\}$, by the repeated usage of \eqref{eq:vtx}, we obtain
\begin{align*}
    \beta(I(G)) &\leq \beta(I(G(u_1|\varnothing)))+\beta(I(G(\varnothing|u_1)))\\
    & \leq \beta(I(G(u_1,u_2|\varnothing))+\beta(I(G(u_1|u_2)))+\beta(I(G(u_2|u_1))+\beta(I(G(\varnothing|u_1,u_2)))\\
    & \leq \cdots  \leq \sum_{X \subset U} \beta(I(G(X|U\setminus X))).
\end{align*}

By Theorem~\ref{Erdos-Posa}, we can take a vertex subset $W$ with $|W| \leq g_3(k)$ hitting all ternary cycles of $G$.
Then we have
\[\beta(I(G)) \leq \sum_{X \subset W} \beta(I(G(X|W \setminus X))).\]
On the other hand, since $N_G[X]\cup(W \setminus X) \supset W$, the graph $G(X|W \setminus X)=G-N_G[X]-(W\setminus X)$ is a ternary graph.
Then by Theorem~\ref{thm:graphcase}, we have $\beta(I(G(X|W \setminus X))) \leq 1$ for any $X \subset W$, and hence we obtain 
\[\beta(I(G)) \leq \sum_{X \subset W} \beta(I(G(X|W \setminus X))) \leq 2^{|W|}=2^{g_3(k)}=2^{O(k \log k)}\]
as desired.
\end{proof}

By applying Lemma~\ref{lem5}, Lemma~\ref{lem6} and Proposition~\ref{prop1}, we can obtain the following hypergraph version as a consequence of Theorem~\ref{thm:graphcase-general}.
\begin{corollary}\label{thm:hypergraph-general}
For all hypergraph $H$ with $t(H) \leq k$, we have $\beta(I(H)) \leq 2^{O(k \log k)}$.
\end{corollary}

\begin{remark}
A trivial lower bound for Theorem~\ref{thm:graphcase-general} (also for Corollary~\ref{thm:hypergraph-general}) is $2^k$ because we have $\beta(I(G))=2^k$ when $G$ is a disjoint union of $k$ ternary cycles.

In general, if $G=G_1 \dot{\sqcup} G_2$, i.e. $G$ is a disjoint union of two graphs $G_1,G_2$, then it is known that $I(G)= I(G_1)*I(G_2)$ where the \textit{join} of two simplicial complexes $X,Y$ are defined as follows:
\[X*Y=\{\sigma \cup \tau \mid \sigma \in X, \tau \in Y\}.\]
Furthermore, by the K\"{u}nneth formula, it is also known that 
\[\tilde{\beta}_{r+1}(X*Y)=\sum_{i+j=r}\tilde{\beta}_i(X)\tilde{\beta}_j(Y)\]
for all $r$, and it implies $\beta(X*Y)=\beta(X)\beta(Y)$.
Hence we obtain $\beta(I(G))=\beta(I(G_1)) \beta(I(G_2))$ when $G=G_1 \dot{\sqcup} G_2$.

Thus if $G=\dot{\bigsqcup}_{i=1}^{k} C_i$ where each $C_i$ is a ternary cycle, then we obtain
\[\beta(I(G))=\beta(I(C_1)) \beta(I(C_2)) \cdots \beta(I(C_k))=2^k\]
since we know that $\beta(I(C))=2$ for every ternary cycle $C$.
\end{remark}

\section*{Acknowledgement}
The author would like to thank Sang-il Oum for his help in finding the current best known bound on the Erd\H{o}s--P\'{o}sa theorem for cycles of length $0$ modulo $m$, and Minki Kim for helpful discussions throughout this research.
The author also thanks the anonymous referees for their valuable comments, which improved the presentation of this manuscript.

\section*{Funding}
This research was supported by a grant(RS-2025-00558178) from the National Research Foundation of Korea(NRF), funded by the Ministry of Science and ICT(MSIT) of the Korean Government and by Global-Learning \& Academic research institution for Master’s · PhD students, and Postdocs(LAMP) Program of the National Research Foundation of Korea(NRF) grant funded by the Ministry of Education(No. RS-2024-00442775), and by the Institute for Basic Science (IBS-R029-C1).

\end{document}